%% file: main.tex
\documentclass[reqno,a4paper,11pt]{article}
\input{header.tex}

\usepackage{geometry}
\usepackage{scrextend}

\usepackage[T1]{fontenc}
\usepackage{titlesec}

\titleformat{\section}{\centering\bfseries\scshape\Large}{\thesection}{1em}{}
\titleformat{\subsection}{\bfseries\scshape\large}{\thesubsection}{1em}{}

\newcommand{\er}{Erd\H{o}s--R\'{e}nyi}
\newcommand{\expander}{$(=\hspace{-0.2em}k,d)$-vertex expander}

\begin{document}

\title{\textsc{\bfseries Finding long cycles in percolated expander graphs}}

\author{\textsc{Lawrence Hollom}\footnote{\href{mailto:lh569@cam.ac.uk}{lh569@cam.ac.uk}, Department of Pure Mathematics and Mathematical Statistics (DPMMS), University of Cambridge, Wilberforce Road, Cambridge, CB3 0WA, United Kingdom}}

\date{}

\maketitle

\begin{abstract}
    Given a graph $G$, the percolated graph $G_p$ has each edge independently retained with probability $p$.
    Collares, Diskin, Erde, and Krivelevich initiated the study of large structures in percolated single-scale vertex expander graphs, wherein every set of exactly $k$ vertices of $G$ has at least $dk$ neighbours before percolation.
    We extend their result to a conjectured stronger form, proving that if $p = (1+\varepsilon)/d$ and $G$ is a graph on at least $k$ vertices which expands as above, then $G_p$ contains a cycle of length $\Omega_\varepsilon(kd)$ with probability at least $1-\exp(-\Omega_\varepsilon(k/d))$ as $k\rightarrow\infty$.
\end{abstract}


\section{Introduction}
\label{sec:intro}



For a graph $G$ and a probability $p\in [0,1]$, the \emph{percolated} graph $G_p$ is formed by retaining each edge of $G$ independently with probability $p$.
The most general form of the question we consider here is: what are the threshold values of $p$ at which various structures emerge in the graph $G_p$?
And what conditions must be placed on $G$ for this to occur?
When the underlying graph $G$ is the complete graph $K_n$, this model of percolation recovers the well-known \er\ random graph $G(n,p)$, and in this situation questions of phase transitions have been subject to a long line of research, beginning with the seminal 1959 paper of Erd\H{o}s and R\'{e}nyi \cite{ER59}.
However, while the structure of the \er\ random graph is generally well-understood, this is not the case for more general percolation.
Indeed, much of the structure of $G$ carries over into $G_p$, and so assumptions on $G$ are required as well; for example, $G_p$ cannot have large connected components if these are not present in $G$.

A natural condition that may be placed on $G$ is that of a minimum degree constraint.
Indeed, Krivelevich and Sudakov \cite{KS13} showed in 2013 by analysis of a depth-first search (DFS) that if $G$ has minimum degree at least $d$, then $G_p$ contains a path of length $\Omega(d)$ with high probability provided that $p = (1+\eps)/d$.
This also provided a new, simple proof of a classic result of Ajtai, Koml\'{o}s, and Szemer\'{e}di \cite{AKS81} concerning long paths in $G(n,p)$.
Note that this is the best value of $p$ that one could hope for, as for $p < 1/d$, one does not in general expect $G_p$ to have even a large component.
They moreover noted many applications of their analysis of a DFS to objects such as digraphs, pseudorandom graphs, and positional games.
A more subtle analysis allowed Krivelevich and Samotij \cite{KS14} to find cycles of length $\Omega(d)$ under the same constraints.

Expansion conditions at a larger scale were considered by Diskin, Erde, Kang, and Krivelevich \cite{DEKK23}, who showed that one may with high probability find components of size $k/2$ in $G_p$ with $p = (1+\eps)/d$ provided that every set of size $k$ has edge-boundary (in $G$) of size at least $dk$.
They moreover noted that this condition does not suffice to find a long path, as, for example, the graph $K_{d,d^{10}}$ has good edge-expansion for sets of size up to $d^9$ but no path of length longer than $2d$.
A way around this issue was suggested by Collares, Diskin, Erde, and Krivelevich \cite{CDEK25}, who considered vertex-expansion rather than edge-expansion.
Indeed, they proved the following result.

\begin{theorem}[\cite{CDEK25}, Theorem 2]
    \label{thm:paths}
    Let $k,d\in \NN$, and let $G$ be a graph on at least $k$ vertices such that every $S\sseq V(G)$ with $\card{S} = k$ satisfies $\card{N(S)} \geq kd$.
    Let $\eps > 0$ be a sufficiently small constant and let $p = (1 + \eps)/d$.
    Then $G$ contains a path of length at least $\eps^2 kd / 10$ with probability at least $1 - \exp(-\Omega_\eps(kd))$.
\end{theorem}

The asymptotic notation $\Omega_\eps$ above is taken as $k\to\infty$, and the subscript $\eps$ indicates that the implicit constant depends on $\eps$.
We use $N(S)$ to denote the set of points in $V\setminus S$ which are adjacent to some element of the set $S$.
Moreover, it may be noted from the proof of \Cref{thm:paths} that the precise conditions on the value of $\eps$ are that $\eps \in (Cd^{-1}, 0.1)$ for an absolute constant\footnote{The lower bound on $\eps$ in terms of $d$ and the constant $C$ was not made explicit in \cite{CDEK25}, but follows easily from the proof. Alternatively, one can consider \Cref{thm:paths} in the regime where $\eps < 0.1$ is fixed and $d$ is taken sufficiently large.} $C$.

The notion of vertex-expansion in \Cref{thm:paths} differs a little from the standard form of vertex expansion.
Indeed, the standard definition is that a $(k,a)$-vertex expander is a graph wherein every $S\sseq V(G)$ with $\card{S}\leq k$ has $N(S)\geq a\card{S}$.
For ease of discussion, we make the following definition.

\begin{definition}
    A graph $G$ is an \emph{\expander} if every $S\sseq V(G)$ with $\card{S} = k$ has $\card{N(S)} \geq dk$.
\end{definition}

It was conjectured in \cite{CDEK25} that, under the assumptions in \Cref{thm:paths}, there should also be a cycle of length $\Omega(kd)$.
Indeed, our main result is the following, which resolves this conjecture.

\begin{theorem}
    \label{thm:main}
    Fix $d > 1$ and $\eps \geq \eps_0(d) > 0$, let $k\in \NN$  and let the graph $G$ be an \expander.
    Then, if $p=(1 + \eps)/d$, then $G_p$ contains a cycle of length at least $\Omega_\eps(kd)$  with probability at least $1 - \exp(-\Omega_\eps(k/d))$.
\end{theorem}

We will deduce this result from \Cref{thm:main-technical}, which shall give more precise information about the dependencies and constants involved.
Where relevant, we use the notation in \cite{CDEK25}.
For a broad introduction to random graphs, see \cite{Bol98,FK15}.

Our proof of \Cref{thm:main} will be outlined in \Cref{subsec:outline}.
Then, in \Cref{sec:preliminaries} we provide a description of the modified DFS algorithm we will employ, along with the standard probabilistic bounds that we will make use of.
The proof of \Cref{thm:main} is contained in \Cref{sec:proof}.
Finally, in \Cref{sec:conclusion} we suggest some directions for future work.


\subsection{Proof outline}
\label{subsec:outline}

As is the case with much of the work in this area, the proof of \Cref{thm:main} runs by careful analysis of a modified depth-first search (DFS) algorithm.
The precise details of the algorithm will be discussed in \Cref{subsec:dfs}, but the idea of the proof is as follows.

We will use sprinkling, writing $p = 1 - (1 - p_1) (1 - p_2)$ with $p_1 \approx p$ and $p_2$ small.
We expose edges with probability $p_1$ for running the DFS, and then expose edges in a second round with probability $p_2$.
This second round will reveal a few more edges, with which we are likely to find a long cycle.
Indeed, we seek an edge $xy$ of $G$ such that there is a long path between $x$ and $y$ in $G_{p_1}$; in fact, we find many of these edges, so that it is likely that one of them is present in $G_{p_2}$.

The DFS algorithm produces a forest in $G_{p_1}$; following the techniques of \cite{CDEK25} we know that this forest $F$ must include a long path $A$.
Indeed, we call an edge $xy$ of $G$ \emph{long} if there is a path of length at least $\alpha kd$ from $x$ to $y$ in $G_{p_1}$.
We will show that, under the assumption that there are few long edges, the path $A$ is likely to grow at (at least) a constant rate $\gamma \approx \eps(1-p)/2$, which will yield a contradiction as the path cannot continue growing indefinitely inside the finite graph $G$.

To be more precise, we will process the vertices in ``blocks'' of $k$ vertices, and the fact that the path continues growing will follow by an induction, which we now outline.
While the path grows linearly, we will be able to effectively bound the number of vertices close to the end of the path, and show that the final block has few neighbours amongst the processed vertices in previous blocks (to avoid producing a long edge).
Thus many neighbours (in $G$) of the final block of $k$ vertices of the path will have to be in the unexplored set $U$, and there will be enough of these that, with high probability, many of them are also neighbours in $G_{p_1}$.

However, we have no control over the size of the graph $G$, which may in particular be far larger than $k$.
Thus we cannot hope for every batch of $O(k)$ consecutive queries in the DFS algorithm to have roughly the expected number of positive answers, even as $k$ grows.
To deal with the situations in which we get far fewer positive responses to queries than expected (resulting in a ``bad'' block), we reserve a collection of ``safe'' vertices, which are spread far apart along $A$. 
When we run into a bad block, we abandon the end of the path, and return to a safe vertex, restarting the DFS algorithm from there.
We show that the probability of a block being bad is exponentially small, and so the number of bad blocks is far smaller than the number of safe vertices, and so the algorithm can continue.


\section{Preliminary results}
\label{sec:preliminaries}


\subsection{Depth-first search}
\label{subsec:dfs}

We use a modified DFS algorithm in our proof.
While DFS algorithms are standard, our modifications are significant enough that it seems worthwhile to fully describe the algorithm. 
When appropriate, we align our notation with that in \cite{CDEK25}.

The DFS algorithm tracks five sets which update over time, as follows.
\begin{itemize}[itemsep=-0.3em]
    \item $W(t)$, the \emph{processed} vertices,
    \item $A(t)$, the \emph{active} vertices,
    \item $U(t)$, the \emph{unprocessed} vertices,
    \item $S(t)$, the \emph{safe} vertices, and
    \item $T(t)$, the \emph{trashed} vertices.
\end{itemize}

We assume throughout the proof that there is a total order on the vertex set $V$ of $G$, and so we may talk about considering vertices ``in order''.
Both $A$ and $S$ are treated as stacks, following a last-in first-out protocol.
We will refer to the vertex most recently added to $A$ as the \emph{head} or \emph{top} of $A$.
Throughout the discussion of the algorithm, we will say that we \emph{query an edge} $e$ of $G$ to mean that we reveal whether or not $e$ is present in $G_p$.
If $e$ is in $G_p$, we will say that the query gave a \emph{positive} result, or, if $e$ is not present in $G_p$, that it gave a \emph{negative} result.
When the algorithm starts, at $t=0$, all vertices are unprocessed. 
Formally,
\begin{equation*}
    W(0) = A(0) = S(0) = T(0) = \emptyset \quad \text{and} \quad U(0) = V.
\end{equation*}
The DFS process is then run in \emph{steps} indexed by $t$ until both $A(t)$ and $U(t)$ are empty.
Step $t+1$ of the basic algorithm is as follows; we will add some further conditions in due course.

\begin{itemize}[itemsep=-0.1em]
    \item If $A(t)$ is empty and $U(t)$ is non-empty, then move the first vertex $v$ from $U(t)$ to $A(t)$, i.e. $A(t+1) = A(t)\union\set{v}$ and $U(t+1) = A(t)\setminus\set{v}$.
    \item Otherwise, let $v$ be the top vertex of $A(t)$. For each neighbour $u$ of $v$ in $U(t)$ in turn, query the edge $uv$ if it has not already been queried, and continue until either a query returns a positive result or all edges from $v$ to $U(t)$ have been queried.
    \vspace{-0.6em}
    \begin{itemize}[itemsep=-0.1em]
        \item If $vu$ is found to exist for some $u\in U(t)$, then move $u$ from $U$ to $A$, so $U(t+1) = U(t) \setminus\set{u}$ and $A(t+1) = A(t) \union\set{u}$.
        \item If all edges are queried and no positive result is found, then move $v$ from $A$ to $W$, i.e. put $A(t+1) = A(t) \setminus\set{v}$ and $W(t+1) = W(t)\union\set{v}$.
    \end{itemize}
\end{itemize}

The graph of positively queried edges forms a forest; let $F(t)$ be the graph on $V$ with edge set consisting of all edges which have been queried with positive results in steps up to and including $t$.
We will consider the component of $F(t)$ containing $A(t)$ to be rooted at the first vertex in $A(t)$; we may thus consider vertices in $W(t)$ in this component as descendants of vertices in $A(t)$.

Say that a vertex $v$ \emph{is processed at time $t$} if $v\notin W(t-1)$ and $v\in W(t)$.
Let $t_i$ be the first time at which precisely $ik$ vertices have been processed (noting that only one vertex can be processed at each time step).

When running queries, for ease of analysis we will use the following method of sampling.
Let $X_{i,j}$ be independent Bernoulli random variables with parameter $p$ for all $i,j\in\NN$.
Sample $X_{1,1},X_{1,2},\dotsc$ in response to queries in steps $1,2,\dotsc,t_1$, and then $X_{2,1},X_{2,2},\dotsc$ for queries in steps $t_1+1,t_1+2,\dotsc,t_2$, and so on.

We will refer to the vertices processed at times in the range $[t_{i-1} + 1, t_i]$ as a \emph{block}, and let $B_i$ be the set of vertices processed during this time interval. 
Say that block $B_i$ is \emph{good} if it -- or, more precisely, the sequence of random variables $X_{i,1},X_{i,2},\dotsc$ -- satisfies the following property.

\begin{definition}
    \label{def:good}
    For $s\geq 1$, let $Y_{i,s}$ be the number of $X_{i,1},X_{i,2},\dotsc,X_{i,s}$ which equal $1$.
    Then a block $B_j$ is \emph{good} if, for all $s\in[dk/2, dk]$, we have $Y_{i,s} \geq (1+\eps)d^{-1}s$.
    A block is \emph{bad} if it is not good.
\end{definition}

We are now ready to define the final steps of our modified DFS algorithm.

\begin{itemize}[itemsep=-0.1em]
    \item If $t = t_{2ik}$ for an integer $i$, and the previous $2k$ blocks have all been good, then let $x$ be the first vertex in $B_{(2i-1)k}$, and let $y$ be the vertex in $A(t)$ closest to $x$. Add $y$ to $S$, i.e. set $S(t+1) = S(t) \union\set{y}$.
    \item If $t = t_j$ for an integer $j$, and the block $B_j$ just processed was bad, then let $y$ be the top vertex in $S(t)$, which is in $A(t)$. If either $S(t)$ is empty or the top vertex of $S(t)$ is not in $A(t)$, then the algorithm stops and fails. Move all vertices above $y$ in $A(t)$ and their descendants in $W(t)$ into $T(t)$, and remove $y$ from $S(t)$.
\end{itemize}

We will prove in due course that, with high probability, the algorithm does not fail and that, at all times $t$ we have $S(t) \sseq A(t)$.
We will refer to any vertex in $S(t)$ for any $t$ as \emph{safe}, even if we are considering a time at which the vertex has already been removed from $S$.

Before proceeding, we make a few simple observations about the algorithm.

\begin{observation}
    \label{obs:dfs}
    The following points all hold deterministically for all $t$.
    \vspace{-0.5em}
    \begin{itemize}[itemsep=-0.3em]
        \item The set $A(t)$ spans a path in $F(t)$ (which is not necessarily induced in $G_p$),
        \item There are no edges of $G_p$ between $W(t)$ and $U(t)$,
        \item $F(t)$ is a forest on $W(t)\union A(t)\union T(t)$, and
        \item Any path from a vertex of $T(t)$ to a vertex of $A(t)\union W(t)$ must pass through a safe vertex.
    \end{itemize}
\end{observation}


\subsection{Concentration inequalities}
\label{subsec:concentration}

We will use the following Chernoff-type bound on the tail of the binomial distribution.
For more details, see e.g. \cite[Section 27.4]{FK15}.

\begin{theorem}
    \label{thm:chernoff}
    If $X$ is a binomial random variable with mean $\mu > 0$, then for every $0 \leq t \leq \mu/2$,
    \begin{equation*}
        \prob{\abs{X - \mu} > t} < 2 \exp \p[\Big]{-\frac{t^2}{3\mu}}.
    \end{equation*}
\end{theorem}


\section{Proof of Theorem \ref{thm:main}}
\label{sec:proof}

Our next result contains explicit dependencies and constants; we shall use it to prove \Cref{thm:main}.

\begin{theorem}
    \label{thm:main-technical}
    Fix positive real numbers $d > 1$ and $\eps \geq 300d^{-1/2}$, let $k\in \NN$  and let $G$ be a graph on at least $k$ vertices such that every $S \sseq V(G)$ with $\card{S} = k$ satisfies $\card{N(S)} \geq kd$.
    Let $p = (1 + 3\eps) / d$.
    Then, the probability that $G_p$ contains a cycle of length at least $\eps^2kd / 100$ is at least $1 - 2\exp(-\eps^2 k / 200d)$.
\end{theorem}

It is immediate that \Cref{thm:main-technical} implies \Cref{thm:main}. 
Indeed, other than replacing $\eps$ by $3\eps$ (which is for notational convenience), \Cref{thm:main} can be deduced from \Cref{thm:main-technical} by simply hiding the various constants in asymptotic notation.

Throughout this proof, we will run the modified DFS algorithm as described in \Cref{subsec:dfs} with percolation probability $p_1 = (1+2\eps)/d$, reserving a small amount of our probability ($p_2 > \eps / d$) to be exposed at the very end of the proof for sprinkling.
First we prove a concentration result concerning the DFS algorithm, stating that blocks are very likely to be good (as in \Cref{def:good}).

\begin{lemma}
    \label{lem:likely-good}
    The probability that block $j$ is good is at least $1 - dk \exp(-\eps^2k/12)$, and this occurs independently of all other blocks.
\end{lemma}

\begin{proof}
    Firstly, independence follows from the fact that each block samples a different sequence of independent Bernoulli variables, so it remains to prove the given bound for fixed $j$.
    Fix $s\in [dk/2, dk]$, and observe that $Y_{j,s} = X_{j,1} + X_{j,2} + \dotsb + X_{j,s}$ is a binomial random variable.
    Noting that $\EE[Y_{j,s}] = p_1 s = (1+2\eps)d^{-1}s$, we may apply Chernoff's inequality (\Cref{thm:chernoff}) to find that
    \begin{align*}
        \prob{Y_{j,s} < (1+\eps)d^{-1}s} 
        &\leq \prob{\abs{Y_{j,s} - \EE[Y_{j,s}]} > \eps d^{-1} s} \\
        &< 2 \exp\p[\Big]{\frac{-\eps^2 s}{3(1+2\eps)d}} \\
        &< 2 \exp\p[\Big]{\frac{-\eps^2 k}{12}}.
    \end{align*}
    The result then follows by a union bound.
\end{proof}

We now deduce from \Cref{lem:likely-good} that, with high probability, there are very few bad blocks.

\begin{lemma}
    \label{lem:few-bad}
    Let $\lambda = \eps^2/24$, and let $Z_s$ be the number of the first $s$ blocks which are bad.
    Then, with probability at least $1 - \exp(-\lambda k / 8)$, for all $s$ we have
    \begin{equation*}
        Z_s \leq s e^{-\lambda k / 2}.
    \end{equation*}
\end{lemma}

\begin{proof}
    We first show that $Z_s$ is zero for all small $s$, and then use Chernoff's bound to deal with larger values of $s$.
    Note first that, for sufficiently large $k$,
    \begin{equation*}
        dk \exp(-\eps^2 k/12) = dk \exp(-2\lambda k) < \exp(-\lambda k).
    \end{equation*}
    Let $s_0 \defined \exp(3\lambda k / 4)$, and note that
    \begin{equation*}
        \prob{Z_{s_0} = 0} \geq (1 - \exp(-\lambda k))^{s_0} \geq 1 - s_0 \exp(-\lambda k) = 1 - \exp(-\lambda k / 4).
    \end{equation*}
    Thus with probability at least $1 - \exp(-\lambda k / 4)$, none of the first $s_0$ blocks is bad.
    We now consider $s\geq s_0$ and apply Chernoff's bound (\Cref{thm:chernoff}) to obtain
    \begin{align*}
        \prob{Z_s \geq s e^{-\lambda k / 2}} 
        &\leq \prob{\abs{Z_s - \EE Z_s} > s (e^{-\lambda k / 2} - e^{-\lambda k})} \\
        &< 2\exp\p[\bigg]{-\frac{s^2(e^{-\lambda k / 2} - e^{-\lambda k})^2}{3se^{-\lambda k / 2}}} \\
        &< 2 \exp(-s e^{-\lambda k/2} / 6).
    \end{align*}
    We now apply a union bound over all $s\geq s_0$, and find that our failure probability is at most
    \begin{align*}
        \sum_{s\geq s_0} 2\exp(-s e^{-\lambda k/2} / 6)
        &= \frac{2\exp(-s_0 e^{-\lambda k/2} / 6)}{1 - \exp(- e^{-\lambda k/2} / 6)} \\
        &< \frac{24 \exp(-e^{\lambda k / 4} / 6)}{e^{-\lambda k / 2}} \\
        &= 24 \exp\p[\bigg]{\frac{3\lambda k - e^{\lambda k / 4}}{6}},
    \end{align*}
    which tends to 0 as $k\to\infty$, and does so faster than $\exp(-\lambda k / 4)$, completing the proof of the lemma.
\end{proof}

The majority of the work in proving \Cref{thm:main-technical} will come in \Cref{lem:key}.
However, before we can state the lemma we need to introduce some more terminology.
At time $t$, two vertices $x$, $y$ in the connected component of $F(t)$ containing $A(t)$ are connected by a unique path $P_{x,y}(t)$ (as $F(t)$ is a forest); write $x\sim_t y$ if this path exists.
We will define the distance $\rho$ between $x$ and $y$ to be the length of this path:
\begin{equation*}
    \rho(x,y) = \min\set{\card{P_{x,y}(t)} \st t\geq 0 \; \text{ and } x\sim_t y}.
\end{equation*}
By convention, the minimum of the empty set to be infinite.
When we refer to the ``distance'' between two vertices, we mean distance in the sense of $\rho(x,y)$.
Note that, as $A(t)\union W(t)\union T(t)$ is a forest, as soon as $x$ and $y$ are connected by a path, this will be the unique path between then for the rest of time, and so as soon as there is a path from $x$ to $y$ at a time $t$, this is the shortest path between these points in $A(\tau)\union W(\tau) \union T(\tau)$ for all $\tau \geq t$.

Note that if $\rho(x,y) \geq s$ is finite and $G_p$ also contains an edge connecting $x$ and $y$, then $G_p$ has a cycle of length at least $s$.
Indeed, call an edge $xy$ of $G$ \emph{long} if there is a time $t$ at which $\rho(t;x,y)$ is both finite and at least $\alpha kd$ where $\alpha \defined \eps^2/100$.

If we can find many long edges of $G$ then we can apply sprinkling and with high probability one of these edges will be added to our percolated graph, giving us a long cycle.
We may now state our key lemma, which will allow us to find many long edges.
Recall that $t_i$ is the minimum time at which precisely $ik$ vertices have been processed (removed from $A$), the block $B_i$ is the set of vertices processed at times in the interval $[t_{i-1} + 1, t_i]$, and $Z_s$ is the number of the first $s$ blocks that are bad.

\begin{lemma}
    \label{lem:key}
    Assume that $Z_s \leq s e^{-\lambda k / 2}$ (i.e. the conclusion of \Cref{lem:few-bad}) holds deterministically for all $s$, that $G_p$ has fewer than $\eps k$ long edges, fix integer $\l \leq \card{V(G)} / k$, and let $\gamma = \eps(1 - d^{-1}(1+\eps)) / 2$ and $\alpha = \eps^2/100$.
    Let $v_j$ be the final vertex processed from block $B_j$ and let $a_j = \card{A(t_j)}$.
    Then the following hold for all $i$
    \begin{enumerate}[itemsep=-0.1em]
        \item \label{item:safe-count} $\abs{S(t_i)} \geq i/2k - 2Z_i$;
        \item \label{item:containment} $S(t_i) \sseq A(t_i)$ and, in particular, the DFS does not fail;
        \item \label{item:safe-far} If $u,w$ are safe vertices then $\rho(u,w) \geq \gamma k^2$;
        \item \label{item:distances} If $i < j < \l$ and all blocks $B_i,B_{i+1},\dotsc,B_j,\dotsc,B_\l$ are good, and $u\in B_i$, $w\in B_\l$, then the path in $W(t_\l)\union A(t_\l)$ from $u$ to $w$ passes within a distance $k+2$ of $v_j$;
        \item \label{item:safe-region} There are at most $2(\alpha d + 3) \gamma^{-1} k$ vertices in $A(t_i) \union W(t_i) \union T(t_i)$ within a distance $\alpha d k$ of a safe vertex $s \in S(t_i)$;
        \item \label{item:growing} If $B_i$ is good, then $a_i - a_{i-1} \geq \gamma k$.
    \end{enumerate}
\end{lemma}

\begin{proof}
    We use induction on $i$, and treat $i = 0$ as a (vacuous) base case.
    We induct on all parts of the lemma simultaneously, and thus prove that they hold for some fixed $i$ in turn, assuming that all of \cref{item:safe-count,item:containment,item:safe-far,item:distances,item:safe-region,item:growing} hold for smaller values of $i$.
    
    \vspace{0.5em}
    \textbf{\Cref{item:safe-count}.}
    We consider when vertices are added to or removed from $S(t_i)$.
    If the blocks $B_{2jk+1}, B_{2jk+2}, \dotsc, B_{2(j+1)k}$ are all good, then a safe vertex is created, and if any of these are bad, then a safe vertex is not created, and one safe vertex is removed from $S$ for every bad block.
    In particular, each bad block can use up one safe vertex and prevent at most one safe vertex from being created, from which the claimed inequality follows immediately.

    Note that, from \Cref{lem:few-bad}, $i/2k - 2Z_i \geq i(1/2k - e^{-\eps^2k/48}) \geq 0$, and so there is always a safe vertex available whenever one is required.

    \vspace{0.5em}
    \textbf{\Cref{item:containment}.}
    Due to the induction hypothesis on \cref{item:growing}, and the observation that a safe vertex $s$ is only created after $k$ further good blocks have been processed, we see that $s$ is at least $\gamma k^2$ vertices behind the head of $A$ when it is first deemed safe.

    Each block can remove at most $k < \gamma k^2$ vertices from the path $A$ (before adding more again), and so $s$ can never be removed from the path unless it is moved to $T$ due to a bad block being discovered. 
    Noting that, due to \cref{item:safe-far}, other safe vertices are still at distance at least $\gamma k^2$ from the head of the path when a bad block is discovered and a safe vertex is used, \cref{item:containment} follows.

    \vspace{0.5em}
    \textbf{\Cref{item:safe-far}.}
    It follows from \cref{item:growing} that each block extends the length of the path by at least $\gamma k$ vertices, and this is done more than $k$ times (in fact, $2k$ times) between safe vertices being created.
    As discussed in the proof of \cref{item:containment} above, the path must grow by at least $\gamma k^2$ vertices between successive creations of safe vertices, whence this item follows.

    \vspace{0.5em}
    \textbf{\Cref{item:distances}.}
    After $B_j$ is processed, the vertex $v_j$ is at distance 1 from the top vertex of $A(t_j)$; we claim that the path from $u\in B_i$ to $w \in B_\l$ must pass through one of the last $k$ vertices of $A(t_j)$.
    Indeed, notice that if the blocks $B_h$ and $B_{h+1}$ are good then the vertex $x$ at distance $(k+1)$ from the top of $A(t_h)$ is also in $A(t_{h+1})$ as only $k$ vertices are removed from the path while processing $B_{h+1}$.
    Then, due to \cref{item:growing}, $x$ continues to be at distance strictly greater than $k+1$ from the top of the path while good blocks are processed, and so is not removed from the path during this time.
    Therefore, $x\in A(t_\l)$.

    Let $y$ be the topmost vertex of $A(t_h)$ which is also in $A(t_\l)$.
    Then $y$ is at distance at most $k+2$ from $v_j$ and is on the path from $u$ to $w$, as required.

    \vspace{0.5em}
    \textbf{\Cref{item:safe-region}.}
    We know from \Cref{obs:dfs} that any path from $T(t_i)$ to $B_i$ passes through a safe vertex $q\in T(t_i)$ (and note that $q\notin S(t_i)$).
    Due to \cref{item:safe-far} and the fact that $\gamma k^2 > \alpha d k$, there are no other safe vertices within distance $\alpha d k$ of $s$ -- in particular, $\rho(s,q) \geq \alpha dk$ -- and thus no vertices of $T(t_i)$ within this distance.
    It suffices to prove that there are at most $(2\alpha d + 4) \gamma^{-1}k$ vertices of $A(t_i) \union W(t_i)$ within distance $\alpha d k$ of $s$.

    There is an integer $j < i$ such that $s\in B_j$, and the definition of safe vertices implies that some descendant of $s$ in $W(t_i)$ is in $B_{(2\l + 1)k}$ for some $\l$.
    Note that \cref{item:growing} implies that $\abs{(2\l + 1)k - j} < \gamma^{-1}$.
    In particular, all vertices of $B_{2\l k}$ and $B_{2(\l + 1)k}$ are at distance at least $\gamma k^2 - k > \alpha kd$ from $s$, and so it suffices to consider only $B_{2\l k+1},B_{2\l k+2},\dotsc,B_{2(\l+1) k-1}$, all of which are good due to the fact that $s$ is safe.
    
    Consider a vertex $u\in B_{j + h}$ for $h$ satisfying $2\l k+1 \leq j+h \leq 2(\l+1) k-1$, and assume that $\rho(u,s) \leq \alpha kd$.
    We know from \cref{item:distances} and \cref{item:growing} that 
    \begin{equation*}
        \alpha kd \geq \rho(u,s) \geq \rho(v_{j+h},v_j) - 2k \geq (\abs{h}\gamma - 2) k.
    \end{equation*}
    This implies that $\abs{h} \leq (\alpha d + 2) \gamma^{-1}$, and the desired result follows immediately.

    \vspace{0.5em}
    \textbf{\Cref{item:growing}.}
    Due to the expansion of $G$, we know that there is a set $X\sseq V \setminus B_i$ of vertices such that $\card{X} = dk$ and every vertex in $X$ is adjacent in $G$ to some vertex of $B_i$.
    The key observation is that, due to the inductive assumptions that the numbers $a_i$ are growing linearly (with the exception of when bad blocks are encountered), only a bounded number of the vertices of $X$ can lie in the set $W(t_i)\union A(t_i)\union T(t_i)$.

    We first show that at most $3(\alpha d + 3)\gamma^{-1} k$ vertices are within distance $\alpha kd$ of some vertex of $B_i$.
    Indeed, we first consider vertices in $B_i,B_{i-1},\dotsc,B_{i-j}$ under the assumption that these blocks are all good, and then add in the contribution from possible bad blocks.

    If these blocks are all good, then we know that due to a similar argument to that used in proving \cref{item:safe-region} that there are at most $(\alpha d + 3) \gamma^{-1} k$ vertices within distance $\alpha kd$ of $B_i$ amongst those vertices processed since the last bad block.

    Noting that, due to \cref{item:safe-far}, there is at most one safe vertex $s$ within distance $\alpha kd$ of $B_i$, and there were at most $2(\alpha d + 3)\gamma^{-1}k$ vertices within distance $\alpha kd$ of $s$ at some time $t$ whereat $s\in S(t)$, we know that in total there are at most $3(\alpha d + 3) \gamma^{-1}k$ vertices of $A(t_i) \union W(t_i) \union T(t_i)$ at distance at most $\alpha kd$ of $B_i$.
    Therefore, we must have
    \begin{equation*}
        \card{X \inter U(\l)} \geq \p[\big]{d - 3(\alpha d + 3) \gamma^{-1} - \eps}k.
    \end{equation*}
    As all vertices in $B_i$ are by definition processed at times before $t_i$, we know that $B_i \sseq W(i)$, and thus every edge from $B_\l$ to $X\inter U(\l)$ has been queried by the DFS algorithm at  a time between $t_{\l - 1}$ and $t_\l$, and these queries have returned negative.
    Thus, between times $t_{\l - 1}$ and $t_\l$, there were at least $(d - 3(\alpha d + 3)\gamma^{-1} - \eps)k \in [dk/2, dk]$ negative responses to queries.
    Due to \Cref{def:good}, there must therefore have also been at least $(1+\eps)(1 - 3(\alpha + 3d^{-1})\gamma^{-1} - \eps d^{-1})k$ positive responses to queries.

    Noting that every positive response leads to the path $A$ growing by one vertex, we see that
    \begin{equation}
        \label{eq:A-diff}
        \card{A(t_i)} - \card{A(t_{i - 1})} \geq (1+\eps)(1 - 3(\alpha + 3d^{-1})\gamma^{-1} - \eps d^{-1})k - k.
    \end{equation}
    It thus suffices to prove that the right-hand side of \eqref{eq:A-diff} is at least $\gamma k$.
    This sufficient condition can be rearranged to the quadratic inequality
    \begin{equation}
        \label{eq:quadratic}
        \gamma^2 - \eps(1 - d^{-1}(1+\eps)) \gamma + 3(\alpha + 3d^{-1})(1 + \eps) \geq 0,
    \end{equation}
    which has roots at
    \begin{equation*}
        \frac{\eps(1 - d^{-1}(1+\eps)) \pm \sqrt{\eps^2(1 - d^{-1}(1+\eps))^2 - 12(\alpha + 3d^{-1})(1 + \eps)}}{2}.
    \end{equation*}
    As $\gamma = \eps(1 - d^{-1}(1+\eps))/2$, $\alpha = \eps^2 / 100$, and $d^{-1} \leq \eps^2 / 300$, we find that the roots of \eqref{eq:quadratic} are real, and thus the right-hand side of \eqref{eq:A-diff} is at least $\gamma k$, as required.
    This completes the proof of \Cref{lem:key}.
\end{proof}

With \Cref{lem:key} in hand, the proof of \Cref{thm:main-technical} is simple: the graph $G$ is finite, and when the DFS algorithm ends, $A$ is empty.
It is immediate, therefore, that $A$ cannot grow indefinitely, but assertion \ref{item:safe-count} of \Cref{lem:key} together with \Cref{lem:few-bad} show that 
\begin{equation*}
    \abs{A(t_i)} > \abs{S(t_i)} \geq i\p[\Big]{\frac{1}{2k} - 2e^{-\lambda k / 2}} > \frac{i}{4k}
\end{equation*}
for sufficiently large $k$, and so the set of safe vertices (and hence also the path) grow without bounds, which is impossible.
Thus the assumption of \Cref{lem:key} -- that there are at most $\eps k$ long edges -- must at some point fail, and so, with probability at least $1 - \exp(-\eps^2 k / 200)$ (coming from \Cref{lem:few-bad}) there are at least $\eps k$ long edges in $G$.

Applying sprinkling with probability $p_2 \geq \eps d^{-1}$, at least one of these $\eps k$ long edges is present with probability at least
\begin{equation*}
    1 - (1 - \eps d^{-1})^{\eps k} \geq 1 - \exp(-\eps^2 d^{-1} k).
\end{equation*}
Thus we find a cycle of length at least $\alpha kd = \eps^2kd / 100$ with probability at least
\begin{equation*}
    1 - \exp(-\eps^2 k / 200) - \exp(-\eps^2 k / d) \geq 1 - 2 \exp(-\eps^2 k / 200d),
\end{equation*}
as required.

As we saw at the start of \Cref{sec:proof}, with \Cref{thm:main-technical} proved, \Cref{thm:main} follows as well.
\qed


\section{Concluding remarks}
\label{sec:conclusion}

We have demonstrated the existence of long cycles in percolated expanders with probability just above the critical threshold for having a large component.
A question that remains is how small $\eps > 0$ can be while percolation with probability $p=(1+\eps)/d$ still results in long paths or cycles.
Indeed, as stated in the introduction, \Cref{thm:paths} of \cite{CDEK25} gives that, with high probability, $G_p$ has a path of length $\Omega(dk)$ when $\eps> Cd^{-1}$ for a sufficiently large constant $C$.

\begin{question}
    \label{question:eps}
    Fix $d > 1$, let $k\in\NN$, and let the graph $G$ be an \expander.
    What (for fixed $d$ and $k\to\infty$) is the minimal order of magnitude of $\eps_0 = \eps_0(d) > 0$ such that, for all $\eps > \eps_0$ and $p = (1+\eps) / d$, the percolated graph $G_p$ with high probability has a path/cycle of length $\Omega_\eps(dk)$?
\end{question}

In particular, if one can take $\eps_0 \leq 1/(d-1)$ in \Cref{question:eps}, then (assuming that the dependence of the probability with which a path/cycle can be found is not too bad) this would imply the pleasant result that, if every edge of $G$ is coloured independently and uniformly at random from a palette of $d-1$ colours, then with high probability, for any colour, the subgraph consisting of edges given that colour would contain a long path/cycle.

In a different direction, much remains to be understood about what more complex structures occur with high probability in $G_p$.
Note that, due to the fact that $\eps$ is bounded from below by a function of $d$ while $k$ (and thus also $\card{G}$) is sent to infinity, it is not unreasonable to expect more complex structure to emerge in $G_p$.
In particular, we make the following conjecture, which would be a significant strengthening of \Cref{thm:main}.

\begin{conjecture}
    \label{conj:subdivide}
    Fix $d\in\NN$ and let $H$ be a graph of maximum degree 3.
    Let $k\in\NN$, and let $G$ be an \expander.
    Then for $\eps \geq \eps_0(d) = d^{-O(1)}$ and $p = (1 + \eps) / d$, the percolated graph $G_p$ contains a subdivision of $H$ with high probability.
\end{conjecture}

It is plausible that the maximum degree of $H$ in \Cref{conj:subdivide} could even be taken as large as $d$.


\section{Acknowledgements}
\label{sec:acknowledgements}

The author would like to thank B\'{e}la Bollob\'{a}s for his comments on this manuscript which greatly improved the presentation.
The author is funded by the Internal Graduate Studentship of Trinity College, Cambridge.


\bibliographystyle{abbrvnat}  
\renewcommand{\bibname}{Bibliography}
\bibliography{main}


\end{document}

%% file: header.tex
\usepackage{amsmath,amssymb,amsthm}

\usepackage{mathrsfs}

\usepackage{mathtools}
\usepackage{commath}

\usepackage{thmtools}
\usepackage{thm-restate}

\usepackage{cases}
\usepackage{enumitem}
\setlist[enumerate,1]{label=(\roman*)}

\usepackage[pdftex, pdfborderstyle={/S/U/W 0}]{hyperref}
\hypersetup{
    colorlinks=true,
    linkcolor=magenta,
    citecolor=cyan,
}

\usepackage{cleveref}

\usepackage{etoolbox}

\usepackage{comment}

\usepackage[numbers]{natbib}

\numberwithin{equation}{section}

\ifdefined\thmcolor
\declaretheoremstyle[
  shaded={bgcolor=\thmcolor}
]{plain}
\else
\fi

\ifdefined\defcolor
\declaretheoremstyle[
  headfont=\normalfont\bfseries,
  bodyfont=\normalfont,
  shaded={bgcolor=\defcolor}
]{noital}
\else
\declaretheoremstyle[
  headfont=\normalfont\bfseries,
  bodyfont=\normalfont,
]{noital}
\fi

\declaretheorem[style=plain,numberwithin=section,name=Theorem]{theorem}

\declaretheorem[style=plain,sibling=theorem,name=Lemma]{lemma}

\declaretheorem[style=plain,sibling=theorem,name=Conjecture]{conjecture}

\declaretheorem[style=plain,sibling=theorem,name=Question]{question}
\declaretheorem[style=plain,sibling=theorem,name=Observation]{observation}

\declaretheorem[style=plain,numbered=no,name=Theorem]{theorem-n}
\declaretheorem[style=plain,numbered=no,name=Proposition]{proposition-n}
\declaretheorem[style=plain,numbered=no,name=Lemma]{lemma-n}
\declaretheorem[style=plain,numbered=no,name=Corollary]{corollary-n}
\declaretheorem[style=plain,numbered=no,name=Conjecture]{conjecture-n}
\declaretheorem[style=plain,numbered=no,name=Claim]{claim-n}
\declaretheorem[style=plain,numbered=no,name=Fact]{fact-n}
\declaretheorem[style=plain,numbered=no,name=Open Problem]{openproblem-n}
\declaretheorem[style=plain,numbered=no,name=Question]{question-n}
\declaretheorem[style=plain,numbered=no,name=Observation]{observation-n}

\declaretheorem[style=noital,sibling=theorem,name=Definition]{definition}

\declaretheorem[style=noital,numbered=no,name=Remark]{remark-n}
\declaretheorem[style=noital,numbered=no,name=Definition]{definition-n}
\declaretheorem[style=noital,numbered=no,name=Construction]{construction-n}
\declaretheorem[style=noital,numbered=no,name=Example]{example-n}

\newcommand{\defined}{\mathrel{\coloneqq}}

\DeclarePairedDelimiter{\p}{\lparen}{\rparen}

\newcommand{\st}{\mathbin{\colon}}

\undef{\set}
\DeclarePairedDelimiter{\set}{\lbrace}{\rbrace}

\undef{\emptyset}
\newcommand{\emptyset}{\varnothing}

\DeclarePairedDelimiter{\card}{\lvert}{\rvert}

\newcommand{\union}{\mathbin{\cup}}
\newcommand{\inter}{\mathbin{\cap}}

\undef{\mod}
\newcommand{\mod}[1]{\ (\mathrm{mod}\ #1)}

\undef{\abs}
\DeclarePairedDelimiterX{\abs}[1]
  {\lvert}{\rvert}{\ifblank{#1}{\,\cdot\,}{#1}}

\undef{\norm}
\DeclarePairedDelimiterX{\norm}[1]
  {\lVert}{\rVert}{\ifblank{#1}{\,\cdot\,}{#1}}

\DeclarePairedDelimiterX{\inner}[2]
  {\langle}{\rangle}{\ifblank{#1}{\,\cdot\,}{#1},\ifblank{#2}{\,\cdot\,}{#2}}

\DeclareMathDelimiter{\given}
  {\mathbin}{symbols}{"6A}{largesymbols}{"0C}

\DeclareMathOperator{\Prob}{\mathbb{P}}
\DeclarePairedDelimiterXPP{\prob}[1]
  {\Prob}{\lparen}{\rparen}{}
  {\renewcommand{\given}{\nonscript\;\delimsize\vert\nonscript\;\mathopen{}}#1}

\DeclareMathOperator{\Expec}{\mathbb{E}}
\DeclarePairedDelimiterXPP{\expec}[1]
  {\Expec}{\lparen}{\rparen}{}
  {\renewcommand{\given}{\nonscript\;\delimsize\vert\nonscript\;\mathopen{}}#1}

\DeclareMathOperator{\Var}{Var}
\DeclarePairedDelimiterXPP{\var}[1]
  {\Var}{\lparen}{\rparen}{}
  {\renewcommand{\given}{\nonscript\;\delimsize\vert\nonscript\;\mathopen{}}#1}

\DeclareMathOperator{\Cov}{Cov}
\DeclarePairedDelimiterXPP{\cov}[2]
  {\Cov}{\lparen}{\rparen}{}{#1,#2}

\newcommand{\eps}{\varepsilon}
\newcommand{\sseq}{\subseteq}

\let\l\relax
\newcommand{\l}{\ell}

\newcommand{\EE}{\mathbb{E}}

\newcommand{\NN}{\mathbb{N}}

